\documentclass{article}
\usepackage{amsmath}
\usepackage{amssymb}
\usepackage{latexsym}
\usepackage{amsthm}
\usepackage{mathrsfs} 
\usepackage{wasysym}
\usepackage{fancyhdr}
\usepackage{amsfonts}
\usepackage{amssymb}
\usepackage{epsfig}
\usepackage{epstopdf}
\usepackage[normalem]{ulem}
\usepackage{eucal}
\usepackage{tikz}

\newtheorem{thm}{Theorem}[section]
\newtheorem{cor}[thm]{Corollary}
\newtheorem{lem}[thm]{Lemma}
\newtheorem{prop}[thm]{Proposition}
\newtheorem{conj}[thm]{Conjecture}

\title{The number of cliques in graphs covered by long cycles}

\author{Naidan Ji\thanks{School of Mathematics and Statistics, Ningxia University, Yinchuan, Ningxia  750021. Email: jind@nxu.edu.cn.}\; and 
Dong Ye\thanks{Department of Mathematical Sciences and Center for Computational Science, Middle Tennessee State University, Murfreesboro, TN 37132. Email: dong.ye@mtsu.edu. Partially supported by Simons Foundation award no. 359516.}}

\begin{document}

\maketitle
\begin{abstract}
 
Let $G$ be a 2-connected $n$-vertex graph and $N_s(G)$ be the total number of $s$-cliques in $G$. Let $k\ge 4$  and $s\ge 2$ be integers. In this paper, we show that if $G$ has an edge $e$ which is not on any cycle of length at least $k$, then $N_s(G)\le r{k-1\choose s}+{t+2\choose s}$, where $n-2=r(k-3)+t$ and $0\le t\le k-4$. This result settles a conjecture of Ma and Yuan and provides a clique version of a theorem of Fan, Wang and Lv. As a direct corollary, if $N_s(G)> r{k-1\choose s}+{t+2\choose s}$, every edge of $G$ is covered by a cycle of length at least $k$. \bigskip

\end{abstract}

\section{Introduction} 

All graphs considered in this paper are simple. Let $G$ be a graph and and $N_s(G)$ be the total number of $s$-cliques in $G$ (a complete subgraph with $s$ vertices). Particularly, $N_2(G)$ is the number of edges of $G$, which is often denoted by $e(G)$. 
The well-known Erd\H{o}s-Gallai theorem~\cite{EG} states that if a graph with $n$ vertices has no cycle of length at least $k$ where $n\ge k\ge 3$, then $e(G)\le (k-1)(n-1)/2$, which was originally conjectured by Tur\'an (cf. \cite{FS}). The Erd\H{o}s-Gallai theorem was improved by Kopylov \cite{GK} for 2-connected graphs.

Before presenting Kopylov's result, we need some extra notations. Let $H_{n,k,a}$ be an $n$-vertex graph whose vertex set can be partitioned into three sets $A$, $B$ and $C$ such that $|A|=a$, $|B|=n-(k-a)$ and $|C|=k-2a$, where integers $n, k$ and $a$ satisfy $n\ge k\ge 4$ and $k/2>a\ge 1$, and whose edge set consists of all edges between $A$ and $B$, and all edges in $A\cup C$. Note that $H_{n,k,a}$ is 2-connected if $a\ge 2$ and has no cycle longer than $k-1$. For $s\ge 2$, define
\[f_s(n,k,a)=
 {k-a\choose s} +(n-k+a){a\choose s-1}.
\]
Then $N_s(H_{n,k,a})=f_s(n,k,a)$ and, particularly, $e(H_{n,k,a})=N_2(H_{n,k,a})=f_2(n,k,a)$. The following is Kopylov's result for 2-connected graphs.

\begin{thm}[Kopylov, \cite{GK}] \label{thm:Kopylov}
Let $G$ be a 2-connected graph on $n$ vertices, and let $n\ge k\ge 5$ and $t=\lfloor \frac{k-1} 2\rfloor$. If $G$ has no cycle of length at least $k$, then 
\[e(G) \le \max\{f_2(n,k, 2), f_2(n,k,t)\}, \]
and the equality holds only if $G=H_{n,k,2}$ or $G=H_{n,k,t}$. 
\end{thm}

It is worth mentioning that Fan, Lv and Wang~\cite{FLW} proved a result slightly stronger than the above theorem for $n\ge k\ge 2n/3$. Together with a result of a result of Woodall~\cite{W}, it provided an alternative proof of Theorem~\ref{thm:Kopylov}. Recently, a 
clique version of Theorem~\ref{thm:Kopylov} has been proven by Luo~\cite{RL} as follows.

\begin{thm}[Luo, \cite{RL}]\label{thm:Luo}
Let $G$ be a 2-connected $n$-vertex graph, and let $n\ge k\ge 5$ and $t=\lfloor \frac{k-1} 2\rfloor$. If $G$ has circumference less than $k$, then the number of $s$-cliques of $G$ satisfies
\[N_s(G)\le \max\{f_s(n,k,2), f_s(n,k,t)\}.\]
\end{thm}

A stability result of the Theorem~\ref{thm:Luo} is obtained Ma and Yuan~\cite{MY}, which also can be viewed as the clique version of a stability result of Theorem~\ref{thm:Kopylov} given by F\"uredi, Kostochka and Verstra\"ete~\cite{FKV}.


Another result of Erd\H{o}s and Gallai  in~\cite{EG} shows that a graph without a path of length at least $k$ has $e(G)\le n(k-2)/2$. The result of Erd\H{o}s and Gallai for paths was strengthened by Fan for 2-connected graphs (Theorem 5 in \cite{Fan90}), which states that the longest path between any pair of vertices in a 2-connected graph with more than $(k+2)(n-2)/2$ edges has length at least $k$. Fan's result is sharp when $n-2$ is divisible by $k-2$, which was further sharpened by Wang and Lv~\cite{WL} for all possible values $n\ge 3$. The sharpness of the results of Fan~\cite{Fan90}, Wang and Lv~\cite{WL} can be shown by the following constructions.

Let $X_{n,k}$ to be an $n$-vertex graph defined as follows. Assuming $n-2=r(k-3)+t$ where $0\le t\le k-4$, the graph $X_{n,k}$ consists of three disjoint parts $A$, $B$ and $C$ such that $A$ is an edge $uv$, and $B$ is a union of $r$ vertex disjoint $(k-3)$-cliques, and $C$ is a $t$-clique, and all edges between $A$ and $B\cup C$. For $s\ge 2$, define 
\[g_s(n,k)=\left \{
\setlength{\jot}{5pt} 
\begin{aligned}
& r{k-1\choose s} +{t+2\choose s} & \mbox{ if }s\ge 3;\\
& r{k-3\choose 2}+{t\choose 2}+2(n-2)+1 & \mbox{ if } s=2.
\end{aligned}
\right.
\]
Then $g_s(n,k)=N_s(X_{n,k})$, and $e(X_{n,k})=N_2(X_{n,k})=g_2(n,k)\le r{k-1\choose 2}+{t+2\choose 2}$. 
These graphs $X_{n,k}$ have no cycle containing the edge $uv$ longer than $k-1$. Note that, if $k>n$, then $X_{n,k}$ is a clique and $g_s(n,k)={n\choose s}$. 


\begin{thm}[Fan~\cite{Fan90}, Wang and Lv~\cite{WL} ]\label{thm:Fan}
Let $G$ be a 2-connected $n$-vertex graph with $n\ge 3$. If $G$ has an edge $uv$ such that $G$ has no cycle of length at least $k\ge 4$ containing $uv$. Then
\[e(G)\le g_2(n,k).\]
\end{thm}

In~\cite{MY}, Ma and Yuan made the following conjecture, which can be treated as a clique version of Theorem~\ref{thm:Fan}. As indicated in~\cite{MY}, the conjecture (if it is true) is a key tool to prove a more general stability result of of Theorem~\ref{thm:Luo}. 

\begin{conj}[Ma and Yuan, \cite{MY}]\label{conj:Ma-Yuan}
Let $G$ be a 2-connected $n$-vertex graph with $n\ge 3$ and let $uv$ be an edge in $G$. Let $k\ge4$ and $s\ge 2$ be integers, and let $n-2=r(k-3)+t$ for some $0\le t\le k-4$. If 
\[N_s(G)>r {k-1\choose s}+{t+2\choose s},\]
then there is a cycle on at least $k$ vertices containing the edge $uv$.  
\end{conj}

Note that, the bound of the above conjecture is not the best possible for the case $s=2$ because of Theorem~\ref{thm:Fan} and $g_2(n,k)< r{k-1\choose 2}+{t+2\choose 2}$ for $n\ge k$. 
The following is our main result, which completely settles Conjecture~\ref{conj:Ma-Yuan}.

\begin{thm}\label{thm:main}
Let $G$ be a 2-connected $n$-vertex graph with $n\ge 3$. If $G$ has an edge $uv$ such that  
$G$ has no cycle containing $uv$ of length at least $k\ge 4$,
then the number of $s$-cliques of $G$ with $s\ge 2$ satisfies
\[N_s(G)\le g_s(n,k).\]
\end{thm}

The bound in Theorem~\ref{thm:main} is sharp due to these graphs $X_{n,k}$ constructed above. A direct corollary of Theorem~\ref{thm:main} shows that if the clique number of a graph $G$ is large enough, every edge of $G$ belongs to a long cycle. 

\begin{cor}
Let $G$ be a 2-connected $n$-vertex graph with $n\ge 3$. If
$N_s(G)> g_s(n,k)$ where $k\ge 4$ and $s\ge 2$, then
every edge of $G$ is covered by a cycle of length at least $k$. 
\end{cor}

\section{Preliminaries}

Let $X_{n,k}$ be an $n$-vertex graph defined in the previous section, and let $Q_{n,k}$ be the $n$-vertex graph with $n\ge 2$ obtained from $X_{n+1, k+1}$ by contracting the edge $uv$ of $A$ into a single vertex $w$. Then for graphs $Q_{n,k}$, it holds that $n-1=r(k-2)+t$ and $0\le t\le k-3$. For $s\ge 2$, define 
\[\psi_s(n,k)=r {k-1\choose s} +{t+1\choose s}.\] 
Then $\psi_s(n,k)=N_s(Q_{n,k})$. Note that, if $k>n$, then $Q_{n,k}$ is a clique and $\psi_s(n,k)={n\choose s}$. 
Let $f_s(n,k,a)$ the function defined in the previous section. By comparing the graphs $H_{n,k,2}$, $H_{n,k,\lfloor \frac{k-1} 2\rfloor}$ and the graph $Q_{n,k}$, it not hard to derive the following proposition. 

\begin{prop} \label{prop}
For integers $n\ge k\ge 5$ and $s\ge 2$, the functions $f_s(n,k,a)$ and $\psi_s(n,k)$ satisfy
\[ \max\{f_s(n,k,2), f_s(n,k, \lfloor \frac{k-1} 2\rfloor)\}\le \psi_s(n,k).\]
\end{prop}\medskip
 
The following result  slightly strengthens Luo's clique version of the Erd\H{o}s-Gallai theorem (Corollary 1.5 in~\cite{RL}), which serves as an important step toward the proof of our main result---Theorem~\ref{thm:main}. Note that, the bound in this result is sharp because of the graphs $Q_{n,k}$ constructed above.

\begin{thm}\label{thm:B}
Let $G$ be a connected $n$-vertex graph with $n\ge 2$. If $G$ has no cycle of length at least $k\ge 4$,  then the number of $s$-cliques with $s\ge 2$ of $G$ satisfies
\[N_s(G)\le \psi_s(n,k).\] 
\end{thm}
\begin{proof}  Let $G$ be a connected $n$-vertex graph with $n\ge 2$. Use induction on $n$, the number of vertices of $G$. The result holds trivially for $n\le 3$.  So assume that $n\ge 4$ in the following, and the theorem holds for all connected graphs with the number of vertices smaller than $n$. 

If $k=4$, every maximal 2-connected subgraph of $G$ is a triangle because the longest cycle of $G$ has length at most $k-1=3$. So each block of $G$ is either a triangle or a single edge. It follows that $N_3(G)\le  (n-1)/ 2$ and equality holds if and only if $G$ is the graph with $(n-1)/ 2$ triangles sharing a common vertex, and $N_2(G)\le \psi_2(n,k)$. Hence
\[N_s(G)\le \psi_s(n,k),\]
and the theorem holds. So, in the following, assume that $k\ge 5$.

First, assume that $G$ is 2-connected. If $n<k$, then 
\[N_s(G)\le {n\choose s}= \psi_s(n,k)\]
and hence the theorem holds. If $n\ge k\ge 5$, then it follows from Theorem~\ref{thm:Luo}  and Proposition~\ref{prop} that
\[ N_s(G)  \le \max\{ f_s(n,k,2), f_s(n,k,\lfloor \frac{k-1} 2\rfloor)\}\le \psi_s(n,k),\]
and the theorem holds.


Hence, we may assume that $G$ has a cut-vertex $v$. Let $H$ be a connected component of $G-v$, and let $G_1=G[H\cup \{v\}]$ and $G_2=G-H$. Then both $G_1$ and $G_2$ are connected and $G_1\cap G_2=\{v\}$. For each $i\in [2]$, let $n_i=|V(G_i)|\ge 2$, and assume $n_i-1=r_i(k-2)+t_i$ with $0\le t_i\le k-3$. Then $n=n_1+n_2-1$. 
Then, for $s\ge 2$, the following inequality holds,
\begin{equation}\label{eq1}
\setlength{\jot}{5pt}
{t_1+1\choose s}+ {t_2+1\choose s}\le 
\left \{ 
\begin{aligned}
& {t_1+t_2+1 \choose s} \mbox{\hspace{3cm} if } t_1+t_2\le k-2;\\
& {k-1\choose s}+{t_1+t_2-k+3\choose s}  \mbox{\hspace{.6cm} if } k-1\le t_1+t_2\le 2(k-3). 
\end{aligned}
\right.
\end{equation}
Applying inductive hypothesis to each $G_i$, we have
\begin{equation*}
\setlength{\jot}{5pt}
\begin{aligned}
\displaystyle N_s(G) & =N_s(G_1)+N_s(G_2) \le \psi_s(n_1, k)+\psi_{s}(n_2, k)\\
                                 &= r_1{k-1\choose s}+{t_1+1\choose s} + r_2{k-1\choose s}+{t_2+1\choose s} \\
                                 &\le \psi_s(n,k),
\end{aligned}
\end{equation*} 
where the last inequality follows from Inequality (\ref{eq1}). This completes the proof. 
\end{proof}

Another ingredient we need to prove Theorem~\ref{thm:main} is {\em edge-switching} operation, which was introduced by Fan~\cite{Fan02} to study subgraph covering. This operation appears in an earlier paper~\cite{AK} of Klemans which studied the probabilities of the number of connected components under this operation.  
 
Let $G$ be a graph and $v$ be a vertex of $G$. Let $N(v)=\{u| uv\in E(G)\}$ and let $N[v]=N(v)\cup \{v\}$. The degree of $v$ in $G$ is denoted by $d_G(v)$ which is equal to $|N(v)|$.  For a given edge $uv$,
an {\em edge-switching} from $v$ to $u$ is to replace each edge $vx$ by a new edge $ux$ for every $x\in N(v)\backslash N[u]$. The resulting graph is called the {\em edge-switching graph} of $G$ from $v$ to $u$, denoted by $G[v\to u]$. 
The following lemma is a trivial observation.

\begin{lem}\label{lem:B}
Let $G$ be a 2-connected graph and let $uv$ be an edge of $G$.\\
{\upshape (i)} If $N(u)\cap N(v)=\emptyset$ and $G/uv$ is not 2-connected, then $\{u,v\}$ is a vertex cut of $G$.\\
{\upshape (ii)} If $N(u)\cap N(v)\ne \emptyset$ and the edge-switching graph $G[v\to u]$ is not 2-connected, then $\{u,v\}$ is a vertex cut of $G$.
\end{lem}

The following lemma shows that the edge-switching operation does not increase the length of longest cycles through certain edges. 

\begin{lem}[Ji and Chen, \cite{JC}]\label{lem:A}
Let $G$ be a connected graph and let $uv$ be an edge. For any edge $ux$, let $k$ be the length of a longest cycle of $G$ containing $ux$. Then the length of a longest cycle containing $ux$ in the edge-switching graph $G[v\to u]$ 
is at most $k$.
\end{lem}

The contraction and the edge-switching operations do not reduce the number of $s$-cliques except small values of $s$ as shown in the following lemma.

\begin{lem}\label{lem:C}
Let $G$ be a connected graph and let $uv$ be an edge of $G$. \\
{\upshape (i)} If $N(u)\cap N(v)=\emptyset$ and $s\ge 3$, then $N_s(G/uv)\ge N_s(G)$;\\
{\upshape (ii)} For $s\ge 2$, it holds that $N_s(G[v\to u])\ge N_s(G)$. 
\end{lem}

\begin{proof} (i) Since $N(u)\cap N(v)=\emptyset$, the graph $G$ has no $s$-cliques containing $uv$ for $s\ge 3$. Hence, every $s$-clique of $G$ remains as an $s$-clique in $G/uv$. Hence (1) follows.

(ii) Let $S(G)$ be the set of unlabeled copies of $K_s$ in a graph $G$. Consider an edge-switching from $v$ to $u$, and let $G'=G[v\to u]$.  Denote $W=\{vx| x\in N(v)\backslash N[u]\}$. 
Define a map $\pi : S(G)\to S(G')$ as follows, 
for each $Q\in S(G)$, 
\begin{equation*}
\pi(Q) =
    \begin{cases}
      Q & \text{if $E(Q)\cap W=\emptyset$,}\\
      Q' & \text{otherwise,}
    \end{cases}      
\end{equation*}
where $V(Q')=(V(Q)\backslash \{v\})\cup \{u\}$ and $E(Q')=(E(Q)\backslash W) \cup \{ux\, |\, vx\in W\cap E(Q)\}$. If $vx\in W\cap E(Q)$, then $x\in N(v)\backslash N[u]$ and it follows that $u\notin Q$.  All neighbors of $v$ in $Q$ are neighbors of $u$ in $Q'$. Hence $Q'$ is indeed an $s$-clique of $G'$. So $\pi$ is well-defined. Note that, $\pi(Q_1)\ne \pi(Q_2)$ for two different $s$-cliques $Q_1$ and  $Q_2$ of $G$. Therefore $\pi$ is an injection. So $N_s(G[v\to u])\ge N_s(G)$ and (ii) holds. 
\end{proof}

\section{Proof of Theorem~\ref{thm:main}}

Now, we are ready to prove our main theorem. Note that Theorem~\ref{thm:main}  follows from Theorem~\ref{thm:Fan} directly  for the case $s=2$. In the following, we only need to prove it for $s\ge 3$.    \medskip

\noindent{\bf Proof of Theorem~\ref{thm:main}.}  Suppose to the contrary that $G$ is a counterexample. For an edge $e$ of $G$, let $c_e(G)$ be the maximum length of cycles containing $e$. Then $G$ is a 2-connected $n$-vertex graph with $N_s(G)> g_s(n,k)$ but does have an edge $e$ such that $c_e(G)<k$. Let
\[ \ell (G)=\max\{d_G(v)| v \mbox{ is an end-vertex of some edge } e \mbox{ with } c_e(G)< k\}.\] 
Among all the counterexamples, choose $G$ such that: (1) the number of vertices of $G$ is as small as possible, and (2) subject to (1), $\ell(G)$ is as large as possible.

Note that the theorem holds trivially for $n=3$.  If $k=4$, then $G$ consists of $n-2$ triangles which share a common edge. Then $N_3(G)=n-2= g_3(n,4)$ and $N_s(G)=0$ for $s\ge 4$, a contradiction to that $G$ is a counterexample. So in the following, assume that $n\ge 4$ and $k\ge 5$. 








\medskip
\noindent{\bf Claim~1.} {\sl The graph $G$ does not  have a 2-vertex cut $\{x,y\}$ such that $xy$ is an edge.}
\medskip

\noindent{\em Proof of Claim~1.} If not, assume that $\{x,y\}$ is a vertex cut of $G$ with $xy\in E(G)$. 
Let $H_1$ a connected component of $G-\{x,y\}$. Further, let $G_1=G[V(H_1)\cup \{x,y\}]$ and $G_2=G-H_1$. 
Then both $G_1$ and $G_2$ are 2-connected.  For convenience, let $|V(G_i)|=n_i\ge 3$, and
$n_i-2=r_i(k-3)+t_i$ and $0\le t_i\le k-4$ for $i\in [2]$. 

If $N_s(G_i)> g_s(n_i,k) $ for some $i\in [2]$, without loss of generality assume $N_s(G_1)>g_s(n_1,k)$. Since $G$ is a counterexample with the smallest number of vertices, the subgraph $G_1$ is smaller and hence not a counterexample. Therefore, $c_e(G_1)\ge k$ for any edge $e\in E(G_1)$. So $c_e(G)\ge c_e(G_1)\ge k$ for each $e\in E(G_1)$. 
For an edge $e'\in E(G_2)$ and $e'\ne xy$, it follows from 2-connectivity of $G_2$ that $G_2$ has a cycle $C'$ containing both $e'$ and $xy$. Let $C$ be a longest cycle of $G_1$ containing $xy$. Then $(C\cup C')-\{xy\}$ is a cycle of $G$ which contains $e'$ and \[c_{e'}(G)\ge |V(C)|+|V(C')|-2> |V(C)|\ge c_{xy}(G_1)\ge k.\]
Thus, $c_e(G)\ge k$ for any edge $e\in E(G)$, a contradiction to that $G$ has an edge $e$ with $c_e(G)<k$. Hence $N_s(G_i)\le g_s(n_i,k)$ holds for both $i=1$ and $i=2$. Therefore, 
\begin{equation}\label{eq2}
\setlength{\jot}{5pt} 
\begin{aligned}
N_s(G) & = N_s(G_1)+N_s(G_2)\le g_s(n_1,k)+g_s(n_2,k)\\
             &= r_1{k-1\choose s} +{t_1+2\choose s} +r_2{k-1\choose s}+{t_2+2\choose s} \\
             &= (r_1+r_2){k-1\choose s}+{t_1+2\choose s}+{t_2+2\choose s}.
\end{aligned}
\end{equation}
For $s\ge 3$, the following inequality holds,
\begin{equation}\label{eq3}
\setlength{\jot}{5pt}
{t_1+2\choose s}+ {t_2+2 \choose s}\le 
\left \{ 
\begin{aligned}
& {t_1+t_2+2 \choose s} \mbox{\hspace{3cm} if } t_1+t_2\le k-4;\\
& {k-1\choose s}+{t_1+t_2-k+5 \choose s}  \mbox{\hspace{.6cm} if } k-3\le t_1+t_2\le 2(k-4). 
\end{aligned}
\right.
\end{equation}
Note that, $n=n_1+n_2-2$ and hence $n-2=(r_1+r_2)(k-3)+(t_1+t_2)$, which implies $r=r_1+r_2$ and $t=t_1+t_2$ if $t_1+t_2\le k-4$, and $r=r_1+r_2+1$ and $t=t_1+t_2-k+3$ if $k-3\le t_1+t_2\le 2(k-4)$. Combining Inequalities~(\ref{eq2}) and (\ref{eq3}), it follows that
\[ N_s(G) \le g_s(n,k).\]
This yields a contradiction to that $N_s(G)>g_s(n,k)$, which completes the proof of Claim~1. 
\medskip

Choose an edge $uv$ of $G$ such that $c_{uv}(G)<k$ and $d_G(v)\le d_G(u)=\ell(G)$.  Then the vertex $u$ satisfies the following claim.

\medskip
\noindent{\bf Claim~2.} {\sl The vertex $u$ is adjacent to all vertices of $G-u$.}
\medskip

\noindent{\em Proof of Claim~2.} Suppose to the contrary that there exists a vertex $w$ in $G$ such that $uw\notin E(G)$. Since $G$ is 2-connected, there is a  $(u,w)$-path $P$ which does not contain $v$. Choose $P=uu_1\cdots u_k w$ to be a shortest $(u,w)$-path avoiding $v$. Then $uu_i\notin E(G)$ for $i\ge 2$. If $N(u) \cap N(u_1)=\emptyset$, then contract the edge $uu_1$ and let $G'=G/uu_1$. By (i) of Lemma~\ref{lem:B} and Claim~1, the graph $G'$ is 2-connected. Since $uu_1\ne uv$, it follows that 
\[c_{uv}(G')\le c_{uv}(G)< k.\] 
Since $G'$ is smaller than $G$,  the graph $G'$ is not a counterexample and hence
$N_s(G')\le g_s(n-1,k)$. It follows from (i) of Lemma~\ref{lem:C}  that
\[N_s(G)\le N_s(G')\le g_s(n-1,k)\le g_s(n,k),\]
which contradicts the assumption $N_s(G)>g_s(n,k)$. 

So assume that $N(u)\cap N(u_1)\ne \emptyset$. Let $G''=G[u_1\to u]$. Then $d_{G''}(u)>d_G(u)$ because $u_2$ and  $w$ are not adjacent to $u$.  By (ii) of Lemma~\ref{lem:B} and  Claim~1, the graph $G''$ is 2-connected. It follows from Lemma~\ref{lem:A} that $c_{uv}(G'')\le c_{uv}(G)< k$. 
Further, by (ii) of Lemma~\ref{lem:C}, it holds that
\[N_s(G'')\ge N_s(G)>g_s(n,k).\]
Then $\ell(G'')\ge d_{G''}(u)> d_{G}(u)=\ell(G)$ because $uu_2\notin E(G)$, which contradicts the maximality of $\ell(G)$.
This completes the proof of Claim~2. 
\medskip


By Claim~2, $u$ is adjacent to all other vertices of $G$. Further, we have the following claim. 

\medskip
\noindent{\bf Claim~3.}  {\sl The graph $G-u$ has no cycle of length at least $k-1$. } 
\medskip

\noindent {\em Proof of Claim~3.} Suppose to the contrary that $G-u$ has a cycle $C=x_1x_2\ldots x_lx_1$ with $l\ge k-1$.
If $v\in V(C)$, let $v=x_1$ (relabelling $x_i$'s if necessary). Then $uvx_2\ldots x_l u$ is a cycle of length at least $k$ containing $uv$ in $G$ because $u$ is adjacent to all other vertices of $G$, a contradiction to $c_{uv}(G)< k$. Now assume $v\notin V(C)$. Since $G$ is 2-connected, the graph $G-u$ is connected. Hence $G-u$ has a path from $v$ to $C$ which is internally disjoint from $C$. Without loss of generality, assume $x_1$ is the end-vertex of $P$ on $C$. Then $uvPx_1x_2\ldots x_l u$ is a cycle of $G$ which has length at least $k$, a contradiction to $c_{uv}(G)<k$. This completes the proof of Claim~3.
\medskip

By Claim~3 and Theorem~\ref{thm:B}, we have
\[N_s(G-u)\le \psi_s(n-1,k-1)= r {k-2 \choose s}+{t+1\choose s},\]
where $(n-1)-1=r(k-3)+t$ and $0\le t\le k-4$.
By Claim~2, it follows that
\begin{equation*}
\setlength{\jot}{5pt} 
\begin{aligned}
N_s(G)&= \displaystyle N_s(G-u)+N_{s-1}(G-u)\\
              &\le \displaystyle r{k-2\choose s}+{t+1\choose s} + r{k-2\choose s-1}+ {t+1\choose s-1}\\
              &= \displaystyle r{k-1\choose s}+{t+2\choose s}\\
              &=g_s(n,k),
\end{aligned}
\end{equation*}       
which yields a desired contradiction to $N_s(G)>g_s(n,k)$. This completes the proof.
\qed

\end{document}